\documentclass[12pt,twoside]{amsart}
\usepackage{amssymb}
\def\Id{{\bf I}}
\def\zbar{{\bar{z}}}

\def\iint{\int }

\def\ID{{\Bbb D}}
\def\IN{{\Bbb N}}

\def\S2{{\bf S}(2)}

\def\IR{{\Bbb R}}

\def\IS{{\Bbb S}}

\def\IN{{\Bbb N}}
\def\IC{\Bbb C}

\def\1ton{1,2,\ldots,n}

\def\Div{{\rm Div}}

\newtheorem{theorem}{Theorem}[section]
\newtheorem{lemma}[theorem]{Lemma}
\newtheorem{corollary}[theorem]{Corollary}

\title[Optimal regularity for planar mappings of finite distortion ]{Optimal regularity for planar mappings of finite distortion}
\author[K. Astala]{Kari Astala}
\address{Department of Mathematics and Statistics, University of
Helsinki,  P.O. Box 68 (Gustaf H\"allströmin katu 2b), FI-00014 University of Helsinki, Finland.}
\email{kari.astala@helsinki.fi}
\author[J. \  Gill]{James  Gill}
\address{ Department of Mathematics,
Campus Box 1146,
Washington University in St. Louis,
St. Louis, MO 63130, Usa}
\email{jgill@math.wustl.edu}
\author[S.\ Rohde]{Steffen Rohde}
\address{Department of Mathematics, University of Washington, C-337 Padelford Hall,
Box 354350,  
Seattle, Washington 98195-4350,  
USA}
\email{rohde@math.washington.edu}
\author[E.\ Saksman]{Eero Saksman}
\address{Department of Mathematics and Statistics, University of
Helsinki,  P.O. Box 68 (Gustaf H\"allströmin katu 2b), FI-00014 University of Helsinki, Finland.}
\email{eero.saksman@helsinki.fi}

\thanks{First author was supported by the Academy of Finland, projects
no.\ 106257, 11064, 1118422 and  211485,  by the Finnish Center of Exellence  
Analysis and Dynamics and  project MRTN-CT-2006-035651, Acronym
CODY, of the European Commission.
Third author was supported by NSF Grant DMS-0501726.
Fourth author was supported by the Academy of Finland, projects
no.\ 113826 and 118765, and by the Finnish Center of Exellence Analysis and Dynamics.}
\subjclass[2000]{Primary 30C62. Secondary 42B20.}
\keywords{Mappings of finite distortion, exponential distortion, optimal regularity, area distortion}
\date{30 January 2008}

\begin{document}

\begin{abstract}
Let $f:\Omega\to\IR^2$ be a mapping of finite distortion, where $\Omega\subset\IR^2 .$ Assume that
the distortion function $K(x,f)$ satisfies $e^{K(\cdot ,f)}\in L^p_{loc}(\Omega )$ for some $p>0.$
We establish optimal regularity and area distortion estimates for $f$. In particular,
we prove that $|Df|^2 \, \log^{\beta -1}(e + |Df|) \in L^1_{loc}(\Omega) $ for every $\beta <p.$
This answers positively, in dimension $n=2$, the  well known  conjectures of Iwaniec-Sbordone  \cite[Conjecture 1.1]{IS} and of Iwaniec-Koskela-Martin \cite[Conjecture 7.1]{IKM}.
\end{abstract}

\dedicatory{We dedicate the paper to the memory of our dear friend and colleague
  Juha Heinonen (1960-2007).}

\input{amssym.def}
\input{amssym.tex}

\maketitle

\section{Introduction} We say that a mapping $f = (u,v) $ defined in a domain $\Omega \subset \IR^2$ is a {\it mapping of finite distortion} if

\begin{itemize}
 \item[(i)] $f  \in W^{1,1}_{loc}(\Omega)$,
 \item[(ii)] $  J(\cdot,f) =  u_x v_y - u_y  v_x  \in L^1_{loc}(\Omega)$, \hskip10pt and 
 \item[(iii)] there is a measurable function $K(z) \geq 1$, finite almost everywhere,  such that
 \begin{equation}\label{eq1} |Df(z)|^2\leq K(z)\, J(z,f) \hskip20pt \mbox{ almost everywhere in $\Omega$}. 
  \end{equation}
\end{itemize}
 The smallest such function is denoted by $K(z,f)$ and is called the distortion function of $f$.
\smallskip

In two dimensions  the mappings of finite distortion are intimately related to elliptic PDE's. For equations with non-smooth coefficients  the first condition, requiring that $f$ has
locally integrable distributional partial derivatives,  is the smallest degree of smoothness where one can begin to discuss what it means to be a (weak) solution to   such an equation.

 The second condition is a (weak) regularity property  which is automatically satisfied by  all homeomorphisms  $f$   in the Sobolev class $ W^{1,1}_{loc}(\Omega)$.  The last condition, that the distortion function $1 \leq K(z,f) < \infty$, merely asks that the pointwise Jacobian $J(z,f) \geq 0$ almost everywhere and that the gradient $Df(z)$ vanishes at those points $z$ where $J(z,f) = 0$. This is a minimal requirement for a mapping to carry   geometric information.

In two dimensions  the distortion inequality  (\ref{eq1}) can be reformulated with the complex notation as
\begin{equation} \label{tuulikki1}
\hskip30pt\left| \frac{\partial f}{\partial \zbar} \right| \leq k(z) \, \left| \frac{\partial f}{\partial z} \right|, \quad \quad  \mbox{ where }   k(z) := \frac{K(z) -1}{K(z) + 1}  < 1,\nonumber
\end{equation}
using  the identifications  $Df = | f_z| + |f_\zbar|$ for the operator norm and $J(z,f) = | f_z|^2 - |f_\zbar|^2$ for   the Jacobian. 
Yet another equivalent formulation  is  with  the  Beltrami equation,
\begin{equation} \label{tuulikki2}
\frac{\partial f}{\partial \zbar} = \mu(z) \frac{\partial f}{\partial z},  \quad \quad \quad  |\mu(z)| = \frac{K(z) -1}{K(z) + 1}  < 1 \; \; \mbox{a.e.},
\end{equation}
where the coefficient  $\mu$ can be defined simply via  the differential equation {\rm (\ref{tuulikki1})}.  Then at a given point $z$ we have  $K(z,f) < \infty$  if and only if  $|\mu(z)|  < 1$. 

The classical theory of quasiconformal mappings considers precisely the mappings of bounded distortion, that is mappings with $K(z,f)  \in L^\infty$ or, equivalently, mappings with complex dilation $| \mu(z)| \leq k $ for some constant  $k  <  1$. 
The classical measurable Riemann mapping theorem tells that if we have a coefficient $| \mu(z)| \leq k <  1$ for $z \in \IC$, then the equation (\ref{tuulikki2}) has a homeomorphic solution and all other solutions are obtained by post-composing with holomorphic functions, see \cite{B}, \cite{A}, or \cite[Chapter 5]{AIM}.

 It is natural to ask how much the condition $K(z,f)  \in L^\infty$ can be relaxed
in order to obtain a  useful theory. 
In the groundbreaking work \cite{David} David  generalized the measurable Riemann mapping theorem to  maps of exponentially integrable  distortion, where in general we  have $\| \mu \|_{\infty} =  1$. Later, Iwaniec,
Koskela, and Martin \cite{IKM} generalized the corresponding regularity theory to higher dimensions, and  for  the two dimensions  a systematic modern development was established by  Iwaniec and Martin as part of  their monograph \cite{IMbook}. During the last 10 years there has been an intensive study of finite distortion mappings, motivated also by the fact that these maps   have natural applications e.g. to non-uniformly elliptic equations and elasticity theory. We refer the reader for instance to  \cite{IMbook},\cite{IM}, \cite{BJ}, \cite{T}, \cite{IKM} or  \cite{FKZ}  and the
references therein, for the  basic literature on the subject.

As a typical example  of a mapping with finite but unbounded distortion  consider the following example\footnote{The  similar  formula (20.75) in  \cite{AIM} has a slight misprint.}. Let
 \begin{equation} \label{kovexample}
g_p(z) = \frac{z}{|z|} \, \left[ \log\left(e + \frac{1}{|z|} \right) \right]^{-p/2} \left[  \log  \log\left(e + \frac{1}{|z|}\right) \right]^{-1/2} \quad \quad  \mbox{ for } |z| < 1,
\end{equation}
and for $|z| >1$ set $g_p(z) =  c_0 \, z.$ 
For each value of the parameter $p > 0$ one easily computes that
\begin{equation} \label{kaava3}
e^{K(z, g_p)} \in L^{p}(\ID) , \nonumber
\end{equation}
while $e^{K(z, g_p)}$ is not locally integrable to any power $s>p.$
Furthermore, we have the fundamental regularity
$ g_p \in W^{1,2}_{loc}(\IC)  \mbox{  when } p >  1.
$
However, for $p=1$ the mapping   $g_1 \notin W^{1,2}_{loc}(\IC)$.

From example (\ref{kovexample}) we see that  the integrability of $e^{K(z)}$ is in general {\it not}  sufficient  for  the  $W^{1,2}_{loc}$-regularity. Instead  we have a  refined logarithmic scale of  regularity  around the space $W^{1,2}_{loc}$. Namely, for each $p > 0$,
$$    |Dg_p|^2 \log^{\beta -1} |Df| \in  L^1_{loc}(\IC), \quad 0 < \beta <  p,
$$
while the  inclusion fails for $\beta = p$.

The exponential integrability of the distortion function in this  theory is no coincidence. Indeed,  to ascertain the  continuity of the mapping $f$ of finite distortion, unless some additional  regularity is required, the exponential integrability  of $K$  suffices but this  condition  has only little room for relaxation, see \cite{IMbook}, Theorem 11.2.1  and \cite{AIM}, Section 20.5.
Hence the mappings of exponentially integrable distortion provide the most natural class where to look for a viable general theory.

 Within this class of mappings one of the  fundamental questions is their  (optimal) regularity. That is, given a homeomorphism with {\it a'priori}  only $ f \in W^{1,1}_{loc}(\Omega)$, we  assume that 
$$e^{K(z,f)} \in L^p_{loc}(\Omega), \hskip20pt \mbox{ where } p   >  0,
$$
and then  seek for  the optimal improvement on regularity this condition brings, i.e.  find in terms of $p$  the best possible Sobolev class where the mapping $f$ belongs to.

The example (\ref{kovexample})  suggests  natural candidates  for the extremal behaviour.   
Indeed, in their paper Iwaniec and Sbordone \cite{IS} formulated a precise general conjecture on the regularity of mappings with  exponentially integrable distortion. The main purpose of this paper is to give a proof to their conjecture, and establish the following theorem.
\medskip

 \begin{theorem}\label{22.hkinov07} Let $\Omega \subset \IR^2$ be a domain. Suppose that the distortion function $K(z,f)$ of  a mapping of finite distortion $ f  \in W^{1,1}_{loc}(\Omega)$ satisfies 
 \begin{equation} \label{keerolta5}
 e^{K(z, f)}  \in L^p_{loc}(\Omega) \quad \mbox{ for some }  p  > 0.
\end{equation}
 Then we have  for every $0 <  \beta <  p$,
 
 \begin{equation} \label{keerolta2}
 |Df|^2 \, \log^{\beta -1}(e + |Df|) \in L^1_{loc}(\Omega) \quad  \mbox{ and }  
\end{equation}
 \begin{equation} \label{keerolta1}
 \hskip-37pt J(z,f) \log^\beta \left(e +  J(z,f) \right)   \in L^1_{loc}(\Omega). 
\end{equation}

\noindent Moreover, for every $p>0$ there are examples that satisfy {\rm (\ref{keerolta5})}, yet fail  {\rm (\ref{keerolta2})} and  {\rm (\ref{keerolta1})}  for $\beta = p$.
 \end{theorem}
 \medskip

 Of particular interest   is the question of when the derivatives  of the mappings are $L^2$-integrable. 
 As a special case of our  theorem one has  
\medskip

\begin{corollary}\label{l2tapaus} Let $\Omega \subset \IR^2$ be a domain and  $ f  \in W^{1,1}_{loc}(\Omega)$ a mapping of finite distortion. Suppose that 
\begin{equation}\label{a}e^{K(z, f)}  \in L^p_{loc}(\Omega) \quad \mbox{ for some }  p  > 1.
 \end{equation}
Then $ f \in W^{1,2}_{loc}(\Omega)$.
 
 The family of mappings  {\rm (\ref{kovexample})} shows that $W^{1,2}_{loc}$-regularity  fails  in general at $p=1$.
  \end{corollary}
  \medskip

 \noindent Previously \cite{IKM}, \cite{IS} it was known that  there exists some  constant  $p_0 >1$  so that $e^{K(z,f)} \in L^p_{loc}$ with   $ p > p_0$ implies   $f \in  W^{1,2}_{loc}$. Our result identifies the precise regularity borderline $p_0 = 1$. Similarly \cite{IMbook} (see \cite{IPMS}, \cite{FKZ} for  $n$-dimensional results) the conclusions {\rm (\ref{keerolta2})} and {\rm (\ref{keerolta1})} were only known to hold 
 under the stronger assumption  $\beta < c_1\, p,$ where $c_1 \in (0,1)$ is some unspecified  constant. 

Our proof of optimal regularity employs the approach  of David \cite{David}. It turns out that it is practically impossible to describe our proof without delving in depth into David's mapping theorem. Hence, to make the exposition as clear as possible we will include a straightforward proof of this result, and present this  in the modern framework developed by Iwaniec and Martin \cite{IMbook}:

\begin{theorem}\label{Existence2} {\rm [David, Iwaniec-Martin]}
Suppose the distortion function $K=K(z)$ is such that $$e^{K}\in L^p(\ID) \quad \mbox{ for some } p > 0.$$
Assume also that $\mu(z) \equiv 0$ for $|z| > 1$.  
Then the Beltrami equation 
 $ f_\zbar(z)=\mu(z) \; f_z(z)$ admits a unique   principal solution $f$ for which
\begin{equation} \label{3varsQ}
f\in W^{1,Q}_{loc}(\IC), \hskip20pt Q(t)=t^2 \log^{-1}(e+t).\nonumber
\end{equation} 
The principal solution is a homeomorphism.
Moreover,  every other $W^{1,Q}_{loc}$-solution $h$ to this  Beltrami equation in a domain $\Omega \subset \IC$ admits the factorization $$h = \phi \circ f,$$ where $\phi$ is a holomorphic function in the domain  $f(\Omega)$.
\end{theorem}
\medskip

\noindent By a principal solution we mean a homeomorphism $ f  \in W^{1,1}_{loc}(\IC)$ that satisfies
\begin{equation} \label{degenerate2}
 f(z)=z+\frac{a_1}{z}+\frac{a_2}{z^2}+\cdots \end{equation}
outside some compact set.

\bigskip
 
There is  a simple explanation  why it is convenient  to develop the theory of mappings of exponential distortion using the space $W^{1,Q}_{loc}$ as a starting point. 
Namely, suppose $f \in W^{1,1}_{loc}(\IC)$ is an orientation preserving homeomorphism, that is  $J(z,f)\geq 0$,  whose distortion function $K(z)$ satisfies $e^K \in L^p_{loc}(\IC)$ for some $ 0 < p < \infty$.  The Jacobian of any planar Sobolev homeomorphism  is locally integrable. Hence we may use the elementary inequality
\[
ab\leq a \log(1+a) + e^b-1
\]
to find that
\begin{eqnarray*} \frac{|Df|^2}{\log(e+|Df|^2)} &\leq&  \frac{KJ}{\log(e+KJ)} 
 \leq \frac{1}{p} \; \frac{J}{\log(e+J)} \; p K
   \; \;  \leq  \; \;  \frac{1}{p}\left( J + e^{pK}-1\right)
\end{eqnarray*}
for all $p>0$.  Thus
\begin{equation}\label{D17.72}
\int_\Omega \frac{|Df|^2}{\log(e+|Df|)} \leq \frac{2}{p}\int_\Omega J(z,f) + \frac{2}{p}\int_\Omega
[e^{pK(z)}-1]\; dz
\end{equation}
for any bounded domain.  This shows that $f$ automatically belongs to the Orlicz-Sobolev class $W^{1,Q}_{loc}(\IC)$.
      
 Conversely  in  \cite{IMbook}, Theorem 7.2.1  it is shown  that for an orientation preserving mapping $f$  the $L^2\log^{-1}L$--integrability of the differential $Df$  implies that  $J(z,f) \in L^1_{loc}(\Omega)$.   This slight gain in the regularity of the  Jacobian determinant is precisely what makes the space $W^{1,Q}_{loc}$  the natural framework for the theory of mappings of exponentially integrable distortion.  
\bigskip

The structure of this paper is the following. First in Section \ref{modcont} we recall the basic 
modulus of continuity estimates for mappings of finite distortion. Section
\ref{neumannkasvu} presents the crucial sharp  estimates for the rate of decay of the Neumann series
of the solution to the  Beltrami equation, in the case of an exponentially integrable distortion. Theorem \ref{Existence2}
is then proven in Section \ref{above}, where our  presentation is  self-contained, apart from the generalized uniqueness and
Stoilow factorization theorem due to Iwaniec and Martin (see Theorem \ref{genstoilow} below). The proof of our main result, Theorem \ref{22.hkinov07} is the content of Section \ref{optimal}. In that section we also  obtain optimal area distortion estimates. The optimality of our  results on the rate of convergence of   the Neumann series
is considered in Section \ref{div}. Finally, Section \ref{elliptic} gives  an application
of the optimal regularity results to degenerate elliptic equations.

 \section{Modulus of Continuity.} \label{modcont}

There is a particularly elegant geometric approach to obtaining modulus of  continuity estimates for functions that are
 monotone. We will not consider the notion of monotonicity in its full generality, but just say that a continuous function $u$  in a domain $\Omega$ is {\it monotone}, if for each relatively compact subdomain $\Omega' \subset \Omega$ we have
 $$  \max_{\partial \Omega'} u =  \max_{ \Omega'} u \quad \mbox{ and } \quad  \min_{\partial \Omega'} u =  \min_{ \Omega'} u,
 $$
that is, if $u$ satisfies both the maximum and minimum principle. For instance, coordinates of open mappings are monotone.
 
 The idea of the next well known argument  goes back to Gehring  in his
study of the Liouville theorem in space.  

\begin{lemma}\label{cc26}
Let $u\in W^{1,2}(2\ID)$ be a continuous and monotone function.  Then for
all
$a,b\in
\ID$ we have
\begin{equation} \label{cc}
|u(a)-u(b)|^2 \leq \frac{\pi \int_{3\ID} |\nabla u|^2}{\log\left(e+\frac{1}{|a-b|}\right)}.
\end{equation}
\end{lemma}
\begin{proof}  Consider the two concentric disks $\ID(z,r)$ and $\ID(z,1)$,  where $z=\frac{1}{2}(a+b)$ and
$r=\frac{1}{2}|a-b|$.  Note that all disks $\ID(z,t)$ with $t\leq \rho_0 = \max \{1, 2r \} $  lie in $3\ID$.

 We have by monotonicity
\[
|u(a)-u(b)| \leq \frac{1}{2}\int_{\partial \ID(z,t)} |\nabla u|\, |dw|
\]
for almost all $r<t<1$. By  Cauchy-Schwartz 
\begin{equation}\label{ccx}
\frac{|u(a)-u(b)|^2}{  t} \leq \frac{\pi}{2 }\int_{\partial
\ID(z,t)} |\nabla u|^2 \, |dw| .\nonumber
\end{equation}
We  integrate with respect to $t$,   $r < t <\rho_0$,  to get 
$$  |u(a)-u(b)|^2 \;  \log(\rho_0 /r) \leq \frac{\pi}{2}  \iint_{\ID(z,\rho_0)}  |\nabla u|^2 dm \leq \frac{\pi}{2}  \iint_{3\ID}  |\nabla u|^2 \,dm .
$$
The estimate (\ref{cc}) is a quick  consequence.
 \end{proof}

In view of this result uniform bounds for the $L^2$-derivatives become valuable. Within the (uniformly elliptic) Beltrami equation such bounds can be obtained for  the {\it inverse} of  the principal solution.
Iwaniec and Sver\'ak \cite{ISv} were the first to make a systematic use of this phenomenon.

\begin{theorem} \label{yksi}
Let $f \in W^{1,2}_{loc}(\IC)$ be the principal solution to $f_{\overline{z}}(z)=\mu(z) f_z(z)$, $ z  \in   \IC$, where $\mu$  is  supported in the unit disk $\ID$ with $\| \mu \|_\infty < 1 $. Let $g = f^{-1}$ be the inverse of $f$. Then 
 \[
 \int_{\IC}  ( |g_{\overline{w}}|^2+|g_w-1|^2 ) dw \leq 2 \int_\ID K(z,f)\, dz .
 \]
\end{theorem}
\begin{proof}
Since the function $g(z)-z$ has derivatives in $L^2(\IC)$, an integration by parts shows that the integral of its Jacobian vanishes. This gives
$$   \int_{\IC}  |g_w-1|^2 -   |g_{\overline{w}}|^2 = \int_{\IC} J(w,g) = 0.
$$
Therefore
\begin{eqnarray}
      \int_{\IC}  |g_w-1|^2 &=& \int_{\IC}  |g_{\overline{w}}|^2 = \int_{f(\ID)}  |g_{\overline{w}}|^2 \leq \int_{f(\ID)}  |D g|^2 \nonumber   \\
     &=&\int_{f(\ID)} K(w,g) J(w,g) =  \int_{\ID}  K\left( f(z), g\right) = \int_{\ID}K(z,f) \nonumber
\end{eqnarray}
since quasiconformal mappings satisfy the usual rules of change of variables.
\end{proof}

 The key fact in the following modulus of continuity estimate, as in Theorem \ref{yksi},    is  that  the bounds do not depend on the value of $ \| \mu \|_\infty$. 
It is obtained easily by combining Lemma \ref{cc26} with Theorem \ref{yksi}.

\begin{corollary}\label{apuva}
Let $f$ and $g= f^{-1}$ be as in Theorem {\rm \ref{yksi}}, and let $R\geq 1$. Then
\begin{equation} \label{varsova}
  |g(a)-g(b)|^2 \leq   \frac{(4\pi)^2 \bigl(R^2+\iint_\ID K(z,f)\; dz\bigr)}{\log\left(e+\frac{1}{|a-b|}\right)} \nonumber
\end{equation} 
whenever $a,b\in \ID(R)$.
\end{corollary}

\section{Decay of  the Neumann Series.} \label{neumannkasvu}

According to the classical measurable Riemann mapping theorem \cite{B}, \cite{A}, given  a compactly supported Beltrami coefficient  with $|\mu(z)| \leq k < 1$ almost everywhere, the system
 \begin{equation} \label{tuulikki3}
 \frac{\partial f}{\partial \zbar} = \mu(z)  \frac{\partial f}{\partial z}, \quad \quad f \in W^{1,2}_{loc}(\IC),\nonumber
\end{equation}
 always admits a homeomorphic solution $f$  with the development (\ref{degenerate2}). That is, $f$ is a principal solution.
 
This uniformly elliptic equation is most conveniently solved  by applying a Neumann-series argument using the Beurling operator ${\mathcal S}$. This is defined by the principal value integral
$${\mathcal S} \varphi (z) :=  - \frac{1}{\pi}\iint_{\IC}\frac{\varphi(\tau)}{(z-\tau)^2} \; d\tau .
$$
The Beurling operator acts unitarily on $L^2(\IC)$.
Also, we have ${\mathcal S}(h_\zbar) = h_z $ for every  $h \in W^{1,2}(\IC)$.  \bigskip
  
  We then look for  the solution in the form
 $$ f = z + {\mathcal C} (\omega),\qquad\qquad C\omega (z):=\frac{1}{\pi}\iint_{\IC}\frac{\omega(\tau)}{(z-\tau)} \; d\tau,
$$
where ${\mathcal C}$ is the Cauchy transform and  $\omega = f_\zbar\,$ is to satisfy the identity
$$ \omega(z) = \mu(z) {\mathcal S} \omega (z) + \mu(z)  \hskip20pt \mbox{ for almost every $z \in \IC .$}
$$
Such an $\omega$ is easy to find when $\|\mu\|_{\infty} \leq k < 1$. We define
\begin{equation}\label{neumann} 
\omega := (\Id-\mu {\mathcal S})^{-1}\mu = \mu+\mu{\mathcal S}\mu +\mu{\mathcal S}\mu{\mathcal S}\mu+\mu{\mathcal S}\mu{\mathcal S}\mu{\mathcal S}\mu+\cdots 
\end{equation}

\noindent The series converges in  $L^2(\IC)$ since  the operator norm of the $n$:th iterate
\[ \| \mu{\mathcal S}\mu{\mathcal S}\cdots \mu{\mathcal S}  \|_{L^{2}(\IC) \to L^{2}(\IC)} \;\; \leq \;\;k^n, \quad n \in \IN  .\] 
The converging series  provides us with the solution $f$.
However, showing that $f$ is a homeomorphism takes some more effort. For this and other basic facts in the classical theory of quasiconformal maps we refer the reader to \cite{A}, \cite{AIM}  and \cite{LV}.

It is natural to study  how far the above method  carries in the  situation of degenerate equations, even if  now   $\|\mu\|_{\infty} = 1$  and  the Neumann series of $ (\Id-\mu {\mathcal S})^{-1}$ cannot converge in the operator norm. Hence the main question here is the rate of decay for  the $L^2$-norms of the terms in the   series (\ref{neumann}).  
  The next result, building  on  the approach initiated by David \cite{David},  gives an optimal answer.

\begin{theorem}\label{2helsinkinov07} 
Suppose  $|\mu(z)| <  1$ almost everywhere, with $\mu(z) \equiv 0$ for $|z| > 1$. If the distortion function $K(z) = \frac{1+|\mu(z)|}{1-|\mu(z)|}$ satisfies
\begin{equation}
e^{K} \in L^p(\ID), \quad \quad p > 0, \nonumber
\end{equation}
then we have for every $0   <   \beta < p$, 
\begin{equation}
\int_{\IC} \left| (\mu {\mathcal S})^n \mu \right|^2 \;  \leq \; C_0 \,   (n+1)^{-\beta}, \quad \quad  n \in \IN, \nonumber
\end{equation}
where   $C_0 =  C_{p,\beta}\cdot  \int_\ID e^{pK} \,$ with $C_{p,\beta}$ depending only on $\beta$ and $p$.
\end{theorem}
\begin{proof}
 Since 
$$   K(z) +1 = \frac{2}{1-|\mu(z)|},
$$
with Chebychev's inequality we have  the measure estimates
\begin{equation} 
 |\{ z \in \ID : |\mu(z)| \geq 1 - \frac{1}{t}  \}| \leq  e^{-2pt} \int_\ID e^{p(K+1)} \, = \, C \,  e^{-2pt}
\nonumber 
\end{equation}
for each $t>1$. With this control on the Beltrami coefficient we will iteratively estimate the terms $ (\mu {\mathcal S})^n \mu$  of the Neumann series. For this purpose we first  fix  the parameter   $0 < \beta < p$, and then for each $n \in \IN$  divide the unit disk into the  "bad" points 
$$ B_n = \{ z \in \ID : |\mu(z)| > 1 - \frac{\beta}{2n+\beta} \}
$$
and the "good" ones, i.e. the complements
$$  G_n = \ID \setminus B_n.
$$
The above bounds on area read now
\begin{equation}
|B_n| \leq  C_1 \,  e^{-4n \, p/ \beta} \, , \quad \quad n \in \IN, \nonumber
\end{equation}
where $C_1 = e^{-p}\int_\ID e^{pK}$.

Next, let us  consider the terms $ \psi_n =  (\mu{\mathcal S})^n \mu \,$  of the Neumann series,  obtained inductively by
\begin{equation} \label{oikeattermit}
\psi_n = \mu {\mathcal S}( \psi_{n-1}), \quad \quad \psi_0 = \mu .\nonumber
\end{equation}
It is helpful to first look at  the following auxiliary terms
\begin{equation} \label{hyvatermit}
g_n =  \chi_{G_n} \, \mu {\mathcal S}( g_{n-1}), \quad \quad g_0 = \mu, \nonumber
\end{equation}
that is, in each iterative step we restrict $\mu$  to the corresponding good part of the disk.

The terms $g_n$ are easy to estimate,
$$\| g_n\|_{L^2}^2 = \int_{G_n} |\mu \,  {\mathcal S}( g_{n-1})|^2 \,  \leq \,  \left( 1 - \frac{\beta}{2n+\beta} \right)^2  \| g_{n-1}\|_{L^2}^2 .
$$
Thus
$$ \| g_n\|_{L^2} \leq   \prod_{j=1}^n  \left( 1 - \frac{\beta}{2j+\beta} \right)  \| \mu \|_{L^2} \,  \leq \,  C_\beta  \, n^{-\beta/2},
$$
where $C_\beta = \sqrt{\pi} \left( 1+ \beta/2 \right)^{\beta/2}$.
For  the true  terms $\psi_n$ decompose 
$$\psi_n - g_n = \chi_{G_n}\mu\,  {\mathcal S}(\psi_{n-1} - g_{n-1}) + \chi_{B_n}\mu  \, {\mathcal S}(\psi_{n-1}).$$   
This  gives the  norm bounds
\begin{equation}\label{joul97}
\| \psi_n - g_n\|_{L^2}^2 \,  \leq \,  \left( 1 - \frac{\beta}{2n + \beta} \right)^2  \| \psi_{n-1} - g_{n-1}\|_{L^2}^2 + R(n),
\end{equation}
where
\begin{equation}\label{joul98}
R(n) =  \| \chi_{B_n}\mu  \, {\mathcal S}(\psi_{n-1})  \|_{L^2}^2 =  \int_{B_n} | (\mu {\mathcal S})^n \mu |^2.
\nonumber
\end{equation}

A main step of the proof is to bound  this last integral term.  A natural approach  is to apply H\"older's inequality and estimate  the norms $\| (\mu{\mathcal S})^n \mu \|_p$ for $p>2$.  This approach can be carried over by applying the spectral bounds on the operator $\mu S$ obtained in \cite{AIS}. On the other hand, these bounds are based on the the area distortion results from \cite{As}. To reveal more clearly the heart of the argument 
  we will  therefore directly use the area distortion. This can be done with a holomorphic representation as follows.
  
 Recall  that the principal solution $f = f^\lambda$ to the equation
\begin{equation} \label{lisakaava7}
  f_\zbar(z)  = \lambda \, \mu(z) f_z(z) 
\end{equation}
  depends holomorphically on the parameter $\lambda \in \ID$. This fact follows from   the Neumann series representation for the derivative,
 \begin{equation} \label{lisakaava8}
  f_\zbar^\lambda = \lambda \mu + \lambda^2 \mu  {\mathcal S} \mu + \cdots + \lambda^n ( \mu  {\mathcal S})^{n-1} \mu + \cdots
\end{equation} 
 The series converges absolutely in $L^2(\IC)$ since $\| \lambda \mu \|_{\infty} \leq |\lambda| < 1$ and ${\mathcal S}$ is an $L^2$-isometry.

The $L^2(\IC)$-valued holomorphic function $\lambda \to  f_\zbar^\lambda $ can as well be represented by the Cauchy integral. 
From this we obtain the following integral representation,
\begin{equation} \label{hotelli1}
\chi_E \,  (\mu{\mathcal S})^n \mu = \frac{1}{2\pi i} \int_{|\lambda| = \rho} \,  \frac{1}{\lambda^{n+2}}  \;  f^\lambda_{\zbar} \, \chi_E \; d\lambda,\quad  \quad E \subset \ID,
\end{equation}
valid for any $0 < \rho < 1$.
We are hence to estimate the norms
$$\|  f^\lambda_{\zbar} \, \chi_E \|_{L^2}^2 = \int_E | f^\lambda_{\zbar} |^2 \leq \frac{|\lambda|^2}{1 - |\lambda|^2} \int_E J(z,f^\lambda ) = \frac{|\lambda|^2}{1 - |\lambda|^2} \,  |f^\lambda (E)|,
$$
 and it is here that the need for a quasiconformal area distortion estimate arises. We have  the uniform bounds   
\begin{equation} \label{tuulikki7}
 |f^\lambda (E)| \leq \pi M |E|^{1/M}, \quad \quad  |\lambda| = \frac{M-1}{M+1},\quad M>1,
\end{equation}
see   \cite{AN}, Theorem  1.6 or    \cite{AIM}, Theorem  13.1.4.

If we now denote  $\rho = |\lambda| = \frac{M-1}{M+1}$ and combine (\ref{hotelli1}) with the above estimates, we end up with the result
\begin{equation} \label{hotelli2}
\|\chi_E \,  (\mu{\mathcal S})^n \mu \|_{2} \, \leq \,  \sqrt{\pi} \left(\frac{M+1}{M-1}\right)^{n}  \frac{M+1}{2}   \, |E|^{1/(2M) .}
\end{equation}
The estimate  is valid for every $M>1$ and for any Beltrami coefficient with $|\mu| \leq \chi_\ID$ almost everywhere. 

For later purposes we recall that the $\partial_z$-derivative, too, has  a  power series representation,
$$f_z^\lambda - 1 = {\mathcal S}f_\zbar^\lambda =   \lambda {\mathcal S} \mu + \lambda^2 {\mathcal S}(\mu  {\mathcal S} \mu) + \cdots + \lambda^n {\mathcal S} ( \mu  {\mathcal S})^{n-1} \mu + \cdots
 $$ 
With an analogous argument  we obtain
\begin{equation} \label{hotelli4}
\| \chi_E {\mathcal S} (\mu{\mathcal S})^n \mu \|_2 \leq  \,  \sqrt{\pi} \left(\frac{M+1}{M-1}\right)^{n+1}  \frac{M+1}{2}   \; |E|^{1/(2M)},
\end{equation}
 an estimate which is similarly valid for every $M>1$.
  \bigskip

Let us then return to estimating  the original Neumann series. Unwinding the iteration in (\ref{joul97}) 
gives
$$ \| \psi_n - g_n\|_{L^2}^2 \,  \leq \,  \sum_{j=1}^n R(j) \prod_{k=j+1}^n \left( 1 - \frac{\beta}{2k+\beta} \right)^2 
\leq 2^\beta  C_{\beta}^{\; 2} \; n^{-\beta} \, \sum_{j=1}^n j^\beta \, R(j) .
$$
It remains to show that the remainder term $R(n)$ decays exponentially. But here we may use the bound (\ref{hotelli2}) for the set $E = B_n$. We have $|B_n| \leq C_1 \,  e^{-4n \, p/ \beta}$ and hence
$$ R(n) \leq 4M^2  \left(\frac{M+1}{M-1}\right)^{2n} \; |E|^{1/M} \leq 4C_1 M^2   \left(\frac{M+1}{M-1}\right)^{2n}  \,  e^{-\frac{4n}{M} \, \frac{ p}{ \beta}}  .
$$
Given $\beta < p$ we can  choose $M >1$ so that 
$$ \log  \left(\frac{M+1}{M-1}\right) - \frac{2}{M} \, \frac{ p}{ \beta} < -  \delta < 0
$$
for some $\delta > 0$. With this choice $R(n) \leq C e^{-2 \delta n} $, $n \in \IN$, where $C = 4M^2 e^{-p} \int_\ID e^{pK}$. This completes the  proof.
\end{proof}
 
\bigskip

According to the Theorem,  the  terms of  the Neumann series decay  with the rate 
$$\|(\mu {\mathcal S})^n \mu \|_2 \leq C_\beta \,  n^{-\beta/2}, \quad \mbox{ for any $\beta < p$}.$$
 If the decay  would be a little better,  of the order   $n^{-\beta}$, then for any $p > 1$  the  series would be norm convergent in $L^2(\IC)$. In view of the example (\ref{kovexample}),  this would immediately determine  $p_0 = 1$ as the critical exponent for the $W^{1,2}$-regularity. 
However, as we will  see later in Section \ref{div}, the  order of decay in  Theorem \ref{2helsinkinov07} cannot be improved. Hence further  means are required for optimal regularity,  to be discussed in  Section \ref{optimal}.

Nevertheless, at the exponent $p_0 = 1$ there is an interesting interpretation of the above bounds,   in terms of the vector-valued   Hardy spaces. Namely, suppose $\mu$ is as in Theorem \ref{2helsinkinov07} and   $p >  \beta > 1$. Then 
\begin{equation} \label{lisakaava11}
  \sum_{n=0}^\infty  \| (\mu {\mathcal S})^n \mu \|^2_{_{L^2}}  < \infty .
\end{equation}
On the other hand, as in  (\ref{lisakaava7}) and (\ref{lisakaava8}),  consider $\lambda \to  f_\zbar^\lambda $ as a  holomorphic vector valued function  $\ID \to L^2(\IC)$. 
With the above bounds  the function has square-summable Taylor coefficients $(\mu {\mathcal S})^n \mu$. In other words, it belongs to the vector-valued Hardy space $H^2(\ID; L^2(\IC))$. 
In particular, since the target Banach space $L^2(\IC)$ has the Radon-Nikodym property, the radial limits 
$$ \lim_{r\to 1}  \frac{\partial f_{r \zeta}}{\partial \zbar} \in L^2(\IC), \quad \quad \zeta \in \IS^1,
$$
exist in $L^2(\IC)$ for almost every $\zeta\in  \IS^1$, see e.g. \cite{Bl}. Similarly for the derivative $\partial_{\zbar} f_{r \zeta}$, and it is not difficult to see that the limiting derivatives are derivatives of a principal solution $f_\zeta$, solving 
\begin{equation} \label{heraajo}
\frac{\partial f_\zeta}{\partial \zbar} =\zeta \,  \mu(z)   \, \frac{\partial f_\zeta}{\partial z} .
\end{equation}
Thus already here we obtain, for $p > 1$ and for almost every $\zeta\in  \IS^1$, the $W^{1,2}_{loc}$-regular solutions to (\ref{heraajo}).
\bigskip

 For exponents  $\beta > 2$, we have via  the norm convergence  (\ref{lisakaava11})  the next auxiliary result, which the reader may recognize as a  distortion of area for mappings in the exponential class.  The result  here is  rather weak but   is a necessary step towards  the optimal  measure distortion bounds which will  be  established in Section \ref{optimal}.

\begin{corollary}\label{1helsinkinov07} 
Suppose $\mu$ and   $0 < \beta < p$ are as in Theorem \ref{2helsinkinov07} and  \begin{equation} \label{vaistyyyo}
\sigma_\mu :=  \sum_{n=0}^\infty  (\mu {\mathcal S})^n \mu  . \nonumber
\end{equation}
If $\beta > 2$, then
$$ \| \chi_E \, \sigma_\mu \|_2 +  \| \chi_E \, {\mathcal S}\sigma_\mu \|_2 \leq C  \log^{1-\beta/2} (e+ \frac{1}{|E|})\, , \quad \quad E \subset \ID,
$$
where $C = C_0 \, C_\beta < \infty$ with $C_\beta$ depending only  $\beta$.

\end{corollary}
\begin{proof}  Theorem \ref{2helsinkinov07} and its proof give us  two different ways to estimate the terms. First,
$$\|\chi_E  (\mu{\mathcal S})^n \mu \|_2  +  \| \chi_E  {\mathcal S}(\mu{\mathcal S})^n \mu \|_2  \leq  2 \| (\mu{\mathcal S})^n \mu\|_2  \leq 2 C_0 \, (n+1)^{-\beta/2} .
$$ 
Summing this up gives
$$ \sum_{n=m+1}^\infty \|\chi_E  (\mu{\mathcal S})^n \mu \|_2  +  \| \chi_E  {\mathcal S}(\mu{\mathcal S})^n \mu \|_2 \,  \leq \,  \frac{4 \, C_0}{\beta - 2} \;  m^{1-\beta/2} \, , \quad \quad m \in \IN .
$$
Secondly, let us choose, say, $M = 3$ in  (\ref{hotelli2}) and  (\ref{hotelli4}). Then for every $n \in \IN$ we have the estimates 
 $ \| \chi_E  (\mu{\mathcal S})^n \mu\|_2 +  \| \chi_E  {\mathcal S}(\mu{\mathcal S})^n \mu \|_2   \leq  6 \sqrt{\pi}  \cdot 2^{n}\,  |E|^{1/6}$.
This, in turn, leads to
  $$ \sum_{n=0}^m \|\chi_E  (\mu{\mathcal S})^n \mu \|_2  +  \| \chi_E  {\mathcal S}(\mu{\mathcal S})^n \mu \|_2 \leq  6 \sqrt{\pi}  \cdot 2^{m+1} |E|^{1/6} .
 $$
 Combining we arrive at
 \begin{equation} \label{lentsikassa}
 \| \chi_E \, \sigma_\mu \|_2 +  \| \chi_E \, {\mathcal S}\sigma_\mu \|_2 \leq  \,  \frac{4 \, C_0}{\beta - 2} \;  m^{1-\beta/2}  + 22\cdot 2^m |E|^{1/6} .\nonumber
\end{equation}
The bound holds for any integer $m \in \IN$, but of course it is   optimal when the two terms on the right hand side are roughly equal. We let
$$ \frac{1}{10}   \log (e+ \frac{1}{|E|}) \leq m < 1 +  \frac{1}{10}   \log (e+ \frac{1}{|E|}),
$$
and with this choice the terms on the right hand side of (\ref{lentsikassa}) are both smaller than $C_0 \,C_\beta 
  \log^{1-\beta/2} (e+ \frac{1}{|E|})$, where the constant  $C_\beta$ depends only on the parameter $\beta > 2$.
 \end{proof}
 
\bigskip

\section{The Degenerate Measurable Riemann Mapping Theorem}\label{above}

We will next give a simple proof of David's theorem, the  generalization of the Measurable Riemann Mapping Theorem to the setting of exponentially integrable distortion. We give this presentation  not only for the reader's convenience, but also since a few  bits and pieces of the  proof will  be needed in the main result, the optimal regularity Theorem \ref{22.hkinov07}.

We start with the case  where the exponential distortion $e^K \in L^p_{loc}$ for  some $p > 2$.  The key idea here is  that for such large $p$, the Neumann series (\ref{neumann}) converges in $L^2(\IC)$, and this quickly yields homeomorphic $W^{1,2}$-regular solutions to the corresponding Beltrami equation (\ref{tuulikki2}). From Section \ref{modcont} we already know the usefulness of such regularity, providing us with general equicontinuity properties. 

After the case of high exponential integrability, together with a few further consequences  established, we then return to the general   measurable Riemann mapping theorem and prove this towards  the end of this section.

\begin{theorem}\label{Existence}
If  $\mu$ is a Beltrami coefficient such that 
\[ |\mu(z)| \leq \frac{K(z)-1}{K(z)+1} \; \chi_\ID  , \] 
where
\begin{equation}
 e^{K} \in L^p(\ID) \hskip15pt {\rm for \,\, some \,\, } p >  2, \nonumber
 \end{equation}
 then the Beltrami equation 
 \[ \frac{\partial f}{\partial \zbar} = \mu(z) \frac{\partial f}{\partial z}\hskip20pt  \mbox{ for almost every $z\in \IC$}  \] 
 admits a unique principal solution $f \in W^{1,2}_{loc}(\IC)$.
\end{theorem}
 
 \begin{proof} 
 Consider the following good approximations  for the Beltrami coefficient $\mu(z)$, 
 \begin{equation}\label{ei19}
\mu_m(z) = \left\{ \begin{array}{ll} 
\mu(z) & \mbox{ if $|\mu(z)|\leq 1-\frac{1}{m}$ } \\
(1-\frac{1}{m})\;\frac{\mu(z)}{|\mu(z)|} & \mbox{ otherwise,} \end{array} \right.
\end{equation}
defined for $m=1,2,\ldots$.  Of course $|\mu_m(z)|\leq 1-\frac{1}{m} < 1$ and we also have the bound
\[
|\mu_m(z)|\leq \frac{K(z)-1}{K(z)+1},
\]
which is independent of $m$.
  Similarly the corresponding distortion functions satisfy
 \[
 K_m(z) = \frac{1+|\mu_m(z)|}{1-|\mu_m(z)|} \leq  \frac{1+|\mu(z)|}{1-|\mu(z)|} = K(z), \;\;\;\;\; m = 1,2,\dots
 \]
 Note that each $ K_m(z) \in L^\infty(\IC)$.
 
 With the classical measurable Riemann mapping theorem, see \cite{B},  \cite{A}, \cite{AIM} or \cite{LV},  we have  unique principal solutions $f^m:\IC \to \IC$ to the Beltrami equation 
$$ (f^m)_\zbar=\mu_m \; (f^m)_z, \hskip20pt m = 1,2,\dots$$
Moreover,
$$ f^m = z + {\mathcal C} (\omega_m),
$$
where $\omega_m = (f^m)_\zbar\,$ satisfies the identity
$$ \omega_m(z) = \mu_m(z) ({\mathcal S} \omega_m)(z) + \mu_m(z)  \hskip20pt \mbox{almost every $z \in \IC$} .
$$
In order for these approximate solutions $f^m$  to converge we need uniform $L^2$-bounds for their derivatives; once these have been established  the proof  follows  quickly. 

We use Theorem \ref{2helsinkinov07}. This gives for any  $2 < \beta  < p $ the estimates
$$
  \| (\mu_m{\mathcal S})^k \mu_m\|_2   \leq  C_m \, k^{-\beta/2}, \quad \quad k \in \IN .
$$
The theorem also gives uniform bounds  for the constant term $C_m$.
Since the approximate solutions have distortion $K_m \leq K$ pointwise, we have
$$C_m  \leq \sqrt{C_{p, \beta}} \left( \int_{\ID} e^{pK_m}\right)^{1/2}  \leq \sqrt{C_{p, \beta}} \left( \int_{\ID} e^{pK}\right)^{1/2}=\sqrt{C_0} ,$$
where $C_0$  is as in Theorem \ref{2helsinkinov07}. Hence if   $e^{K} \in L^p(\ID)$ for some $p>2$, 
 we  obtain $L^2$-estimates for the sums (\ref{neumann}),
 \begin{equation} \label{hotelli11}
\| (f^m)_\zbar \|_2 \leq \sum_{k=0}^\infty  \| (\mu_m{\mathcal S})^k \mu_m\|_2   \leq  \,  \sqrt{C_0}  \sum_{k=0}^\infty (k+1)^{-\beta/2}.
\end{equation}
Hence also  $$\| (f^m)_z - 1 \|_2 = \| (f^m)_\zbar \|_2 \leq \sqrt{C_0} \sum_{k=0}^\infty (k+1)^{-\beta/2}.
 $$
This easily yields that
 \begin{equation} \label{tuulikki4}
 \int_{\ID(R)} |Df^m |^2  \leq 4 \pi R^2 + A, \hskip20pt R > 1,
\end{equation}
where $A < \infty$ depends only on $ \beta>2 $ and $\int_{\ID} e^{pK}$.

Now Lemma  \ref{cc26} applies and shows that on compact subsets of the plane the sequence $\{ f^m \}$ has a uniform modulus of continuity. Changing to a subsequence, and invoking (\ref{tuulikki4})  again we may assume that $f^m(z) \to f(z)$ locally uniformly  with the derivatives $D f^m$ converging weakly in $L^{2}_{loc}(\IC)$.  
 Their weak limit necessarily equals $Df$. Furthermore,  since for any $R>1$ and $\varphi \in C^{\infty}_0(B(0,R))$ 
$$ \int_\IC \varphi \bigl( f^m_\zbar - \mu \, f^m_z \bigr) =  \int_\IC \varphi \bigl( \mu_m - \mu \bigr) \, f^m_z  \leq \frac{  \sqrt{\pi} \,  \| \varphi \|_\infty}{m} \, \sqrt{4 \pi R^2 + A} \to 0
$$
by (\ref{tuulikki4}), we see  that $f$ is a solution to the Beltrami equation
$$ f_\zbar=\mu \; f_z  \hskip20pt  \mbox{for almost every $z \in \IC$}.$$

On the other hand, we can use  Corollary \ref{apuva} for the inverse mappings $g_m = (f^m)^{-1}$ and see that $g^m(z) \to g(z)$ uniformly on compact subsets of the plane, where $g:\IC \to \IC\;$ is continuous. From $g^m\left(f^m(z)\right) = z$  we have $g \circ f(z) = z$. That is, $f$ is a homeomorphism, hence the principal solution we were looking for.

Lastly, for the uniqueness we refer to the following very general Stoilow factorization result from the monograph \cite{IMbook}; see also \cite{AIM}, Section 20.4.8. The proof of Theorem \ref{Existence} is now complete.
 \end{proof}

\begin{theorem} \label{genstoilow} {\rm (\cite{IMbook}, Theorem 11.5.1)}. Let 
\begin{equation} \label{kuu}
 Q(t) = \frac{t^2}{\log(e + t)} 
\end{equation}
and suppose we are given a homeomorphic solution $f \in W^{1,Q}_{loc}(\Omega)$ to the Beltrami equation
\begin{equation}\label{A94}
f_\zbar = \mu(z) f_z, \hskip20pt  \;\;\; z \in \Omega,
\end{equation}
where $| \mu(z)| < 1$ almost everywhere. Then every other solution $h \in W^{1,Q}_{loc}(\Omega)$ to {\rm (\ref{A94})} takes the form
\[  h(z)= \phi(f (z)), \hskip20pt z\in \Omega , \] 
where $\phi:f(\Omega)\to \IC$ is holomorphic.
\end{theorem}
\bigskip

In Theorem \ref{Existence} we constructed  the mapping $f$ through a limiting process rather than  directly by the Neumann series. Nevertheless, the representations
\begin{equation} \label{hotelli12}
 f_\zbar = \sum_{n=0}^\infty (\mu{\mathcal S})^n \mu, \quad \quad f_z -1 = \sum_{n=0}^\infty {\mathcal S}(\mu{\mathcal S})^n \mu
\end{equation}
are still valid. Namely, by Theorem \ref{2helsinkinov07}  the sums are absolutely convergent in $L^2(\IC)$.  Since  the approximate coefficients  $\mu_m \to \mu$ in $L^\infty(\IC)$,  the terms $(\mu_m{\mathcal S})^k \mu_m  \to (\mu {\mathcal S})^k \mu$ weakly in $L^2(\IC)$. As the derivatives of $f^m$ are representable by the corresponding  Neumann series,    the identities (\ref{hotelli12})  follow  with the help of the uniform bounds (\ref{hotelli11}).

In studying the general exponentially integrable distortion we need to know that the mappings of Theorem \ref{Existence} preserve sets of Lebesgue measure zero. This follows easily from the above.
 
\begin{corollary} \label{lennossa3}
 If $\mu$ and the principal solution $f\in W^{1,2}_{loc}(\IC)$ are as in Theorem \ref{Existence}, then $f$ has the properties ${\mathcal N}$ and ${\mathcal N}^{- 1}$, i.e. for any measurable set $E\subset \IC$,
$$ |f(E)| = 0 \quad \Leftrightarrow \quad |E| = 0 .$$
\end{corollary}
\begin{proof}
Any  homeomorphism with the $W^{1,2}_{loc}$-regularity preserves Lebesgue null sets, see \cite{MM},  or   \cite{AIM} for an elementary proof.  We have shown that $f \in W^{1,2}_{loc}(\IC)$, and since $K(\cdot, f) \in L^1(\ID)$,  Theorem \ref{yksi}  verifies that the inverse map $g= h^{-1}$ belongs to $W^{1,2}_{loc}(\IC)$. The claim follows.
 \end{proof}

Concerning the existence of solutions, the situation  is rather different in the general case  where the integrability exponent of $e^K$ is small. Recall that   by
example  (\ref{kovexample})  the corresponding principal solutions need not be in $W^{1,2}_{loc}(\IC)$.
 Instead,  inequality (\ref{D17.72})  leads us to look for
principal solutions $f$ in $W^{1,Q}_{loc}(\IC)$, where $Q(t)$ is given in (\ref{kuu}).
The idea of the following proof is then to reduce the argument  to Theorem \ref{Existence} via a suitable
factorization.
\smallskip

\begin{proof}[Proof of Theorem \ref{Existence2}.] 
Given  a Beltrami coefficient $\mu=\mu (z)$ with the distortion function $K(z)  = \frac{1+ |\mu(z)|}{1-|\mu(z)|}$, assume that $e^K \in L^p(\ID)$ for some $0 <p \leq  2$. It is very suggestive to write  
\[ K(z)=\frac{3}{p}\cdot \frac{p
K(z)}{3} .\]
Then $K_1(z)= p K(z)/3$ satisfies the hypotheses of Theorem \ref{Existence} 
and the constant  factor $K_2(z)= 3/p > 1$ could be  represented as the distortion of a quasiconformal mapping. There is one problem here in that at points where $K(z)$ is already finite and perhaps small,  the factor  $K_1(z)$ might be less than $1$ and so cannot be a distortion function.  We will get around this point by constructing the related Beltrami coefficients  using the  hyperbolic geometry.

Let $M > 1$. For each $z\in\ID$  choose a point $\nu = \nu(z)$ on the radial segment determined by $ \mu(z)$, so that 
\begin{equation}\label{nuchoice}
\rho_\ID(0,\nu) + \rho_\ID(\nu,\mu) \, = \,  \rho_\ID(0,\mu)  \, = \,  \log \frac{1+|\mu|}{1-|\mu|}.
\nonumber
\end{equation}
 If $\rho_\ID(0,\mu) > \log M$ we require $\rho_\ID(\nu,\mu)= \log M$ and  otherwise we set $\nu = 0$.
In any case we always have 
\[ M K_\nu=e^{\log M} e^{\rho_\ID(0,\nu)}  \leq K_\mu + M. \]
It follows that 
\begin{equation} \label{lennossa2}
\int_\ID e^{p  M K_\nu} \leq e^{p M} \int_\ID e^{p   K_\mu} < \infty ,
\end{equation}
so  that $\nu$ satisfies the hypotheses of Theorem  \ref{Existence}  as soon as we choose  $ M = \frac{3}{p}$.  We can therefore solve the Beltrami equation for $\nu = \nu(z)$ to get a principal mapping $F$ of class $W^{1,2}(\ID)$.  Next, set 
\begin{equation}\label{kapdefdeg}
\kappa (w)  = \frac{\mu(z) - \nu(z)}{1-\mu(z){\overline{\nu(z)}}}\; \left(\frac{F_z}{|F_z|}\right)^2, \quad \quad w = F(z), \; z \in \IC. 
\end{equation}
According to Corollary \ref{lennossa3} $\kappa$ is well defined almost everywhere. We also  see that
\begin{equation}
\frac{1+|\kappa|}{1-|\kappa|} = \frac{1+\left|\frac{\mu - \nu}{1-\mu\bar\nu}\right|}{1-\left|\frac{\mu - \nu}{1-\mu\bar\nu}\right|}  = e^{\rho_\ID(\nu,\mu)} \leq M = \frac{3}{p} < \infty .\nonumber
\end{equation}
 Thus we may solve the uniformly elliptic Beltrami equation for $\kappa$ to obtain a $M$-quasiconformal principal mapping $g$.  
 
 We next put $f=g\circ F$. The classical Gehring-Lehto result \cite{GL} on Sobolev homeomorphisms verifies that $F$ is differentiable almost everywhere. Since $F$ has the Lusin property ${\mathcal N}^{\pm 1}$,  the same is true for $f$, and we have  $$|Df(z)|^2  \leq
 M J(w,g) K(z,F) J(z,F) = M K_\nu(z) J(z,f), \quad \quad w = F(z) . $$
 Arguing as in (\ref{D17.72}) we see that 
  \begin{equation}\label{h}
  \frac{|Df|^2}{\log(e + |Df|)} \leq \frac{2}{p} \Bigl[   J(\cdot, f) + e^{p M K_\nu}\Bigr]
\end{equation}
is locally integrable. 

In the proof of Theorem \ref{Existence} we  constructed the approximants $F^m$ of the principal  mapping $F$.  To obtain  uniform bounds  we may  apply the classical Bieberbach area theorem  \cite{Duren}. This  gives
$$\int_{B(0,r)}J(z, \phi) \leq \pi r^2, \quad \quad r\geq 1,$$  
for any quasiconformal  principal solution, conformal outside the unit disc.
In particular with  (\ref{h}) we  see that   the compositions $g\circ F^m$ are uniformly bounded in the space $W^{1,Q}_{loc}(\IC)$, $$Q(t) = \frac{t^2}{\log(e+t)}.$$  Since  $g\circ F^m(z) \to f(z)$  locally uniformly, we infer that $f$ is a Sobolev mapping contained in the class
  $W^{1,Q}_{loc}(\IC)$. 
  However, note that  although the
quasiconformal map $g$ will lie  \cite{As}  in the space $W^{1,s}_{loc}(\IC)$  for all
\[ 2\leq s< \frac{6}{3-p}, \]
yet the composition $f= g \circ F$ will not lie in $W^{1,2}_{loc}(\IC)$ in general.

  We defined the Beltrami coefficient $\kappa = \mu_{g}$ through the formula (\ref{kapdefdeg}), but  as well  one may identify $\mu_g = \mu_{f \circ F^{-1}} $ via the familiar formula for the dilatation of a composition of mappings,
 \begin{equation} \label{kumpula16}
  \mu_{f \circ \, F^{-1}} ( w ) = \frac{ \mu_f(z) - \mu_F(z)}{1 - \mu_f(z) {\overline{\mu_F(z)}}}    
\Bigl(\frac{F_z \,(z)}{|F_z\,(z) |}\Bigr)^2,   \hskip20pt w = F(z),\nonumber
\end{equation}
which is a direct consequence of the chain rule.
Comparing the expressions  shows that $\mu_f = \mu$, and hence $f$ in fact solves the required Beltrami equation $f_\zbar = \mu f_z$.   Clearly $f$ is a principal solution.  Therefore we have established the existence of the measurable Riemann mapping, for any exponentially integrable distortion.
  
   Uniqueness and Stoilow factorization now follow from Theorem  \ref{genstoilow}. Thus the proof of Theorem \ref{Existence2} is complete. 
\end{proof}
 
 Once we have  the existence of the principal solution, with the equicontinuity properties provided by Lemma \ref{cc26} and Corollary \ref{apuva}, the usual normal family arguments quickly give solutions to the global degenerate Beltrami equation, where the coefficient $\mu(z)$ is not necessarily compactly supported. For an overview, see \cite{AIM}, Section 20.5.
 \bigskip
 
 From the above proof and (\ref{lennossa2}) we  distill a powerful factorization, one of  the key facts  in obtaining the sharp regularity.

\begin{corollary} \label{lennossa}
Suppose the distortion function $K=K(z)$ satisfies $e^{K}\in L^p(\ID)$ for some $ p > 0$. Then for any $M \geq 1 $ the principal solution  to  $ f_\zbar(z)=\mu(z) \; f_z(z)$ admits a factorization
$$ f = g \circ F,
$$
where both $g$ and $F$ are principal mappings,  $g$ is $M$-quasiconformal and $F$ satisfies
$$ \int_{\ID} e^{pM K(z,F)}  \leq C_0 < \infty.
$$
\end{corollary}

\medskip

\noindent{\it Remark.} \quad Using the above factorization we obtain  the properties ${\mathcal N}$
and ${\mathcal N}^{-1}$
for all maps of  exponential distortion.
  Namely, assume that $ \phi \in W^{1,Q}_{loc}(\IC)$ is a mapping of finite distortion with
 $e^{K(z, \phi)}  \in L^p_{loc}(\Omega)$ for some $ p  > 0 $. Then first use locally  the Stoilow factorization  $\phi = h \circ f$ where $h$ is holomorphic and $f$ is a principal solution as in  Theorem  \ref{Existence2}. Applying  Corollary  \ref{lennossa} we can make a  further factorization,
  $ f = g \circ F
  $
  where $g$ is quasiconformal and  $F$ has exponential distortion  $e^{K} \in L^p$ with $p > 2$.  According to   Corollary \ref{lennossa3} each  factor  in $\phi = h \circ g \circ F$  has the properties ${\mathcal N}^{\pm 1}$.

 \section{Optimal regularity: proof of Theorem \ref{22.hkinov07}} \label{optimal}

We start with  general area distortion bounds  that have independent interest, since they are optimal up to estimates at the borderline case. In fact,
the factorization method described in Corollary \ref{lennossa} enables us to  improve the  distortion exponent of Corollary \ref{1helsinkinov07}.
 \begin{theorem}\label{2.hkinov07} Let $|\mu(z)| < 1$ almost everywhere and  $\mu(z) \equiv 0$ outside the unit disk. Suppose $ f$ is a principal solution to $f_\zbar = \mu f_z$. If
   $$e^{K(z, f)}  \in L^p(\ID) \quad \mbox{ for some }  p  > 0,
 $$
 then for any $0 <\beta  < p$ we have 
 \begin{equation} \label{hotelli7}
 |f(E)| \leq C \log^{-\beta}(e + \frac{1}{|E|}), \quad \quad E \subset \ID.
\end{equation}
The constant $C$ above depends on $\beta$, $p$ and $\| e^{K(z,f)} \|_p$ only.
  \end{theorem}
 \begin{proof}
 Choose $ \beta_0 \in (\beta, p)$ and $M \geq 1$ so that 
 $$\frac{2}{M} < \beta < \beta_0 - \frac{2}{M} .
 $$ 
 We will then use  the factorization $f = g \circ F$ from Corollary \ref{lennossa}. Since $pM > \beta_0 M > 2$, 
Corollary \ref{1helsinkinov07} applies to $ \sigma_\mu = F_\zbar $ and $F_z = 1 + {\mathcal S}F_\zbar
=1+{\mathcal S} \sigma_\mu$,
 $$ |F(E)| = \int_E |F_z|^2 - |F_\zbar|^2 \leq  2|E| + 2 \int_E |F_z -1|^2 \leq C   \log^{2-\beta_0M}(e + \frac{1}{|E|}).
 $$ 
Above $|F(E)|$ is legitimately obtained by integrating the Jacobian since $F$ is a Sobolev homeomorphism that
 satisfies condition ${\mathcal N}$.
 On the other hand, since $g$ is an $M$-quasiconformal principal mapping we can use the area distortion estimate (\ref{tuulikki7}).
 This gives 
  $$|f(E)| = |g \circ F(E)| \leq \pi^{} M |F(E)|^{1/M} \leq \pi\, C \Bigl[ \log(e +  \frac{1}{|E|}) \Bigr]^{(2-\beta_0 M )/M}.
 $$
 Since $\beta < \beta_0 - \frac{2}{M} $ the result follows.
\end{proof}

Again  the family (\ref{kovexample}) shows  that  the measure distortion bounds are  optimal in terms
of the exponent of the logarithm in (\ref{hotelli7}), up to the possible estimates at  the borderline $\beta = p$. 

We are ready for  

 \begin{proof}[\bf Proof Theorem \ref{22.hkinov07}.] Assume that  we have a mapping of finite distortion $f \in W^{1,1}_{loc}(\Omega)$  with  $e^{K(z,f)} \in L^p_{loc}(\Omega)$. From (\ref{D17.72}) it follows that  $f \in W^{1,Q}_{loc}(\Omega)$, and hence  the general Stoilow factorization Theorem \ref{genstoilow} applies.
 With the  factorization we may in fact assume that $\Omega =\IC$ and that $f$ is the  principal solution of Theorem \ref{Existence2}, conformal outside $\frac{1}{2}\ID$. It is then enough to consider
the regularity over $\ID.$
 
We prove first
\begin{equation} \label{keerolta11}
 J(z,f) \log^\beta \left(e +  J(z,f) \right)   \in L^1_{loc}(\ID) \quad   \quad  0 <  \beta <  p .
\end{equation}
 In order to verify this it is enough to show  that $J^*(x) \log^\beta \bigl(e +  J^*(x) \bigr)   \in L^1(0,\pi)$, where  $J^*(x)$, $x \geq 0$,  is the  nonincreasing  rearrangement of $J(z,f)$.  
 
For the rearrangement the estimate  (\ref{hotelli7})  holds   in the form
\begin{equation} \label{lisakaava21}
   \int_0^t J^*(x)dx \leq C \log^{-\beta}(e + \frac{1}{t}), \quad   \quad  0 <  \beta <  p, \quad  0 < t \leq \pi.
\end{equation}
 Observe that we have applied here  the remark after Corollary \ref{lennossa}. 
 Further, since $J^*$ is nonincreasing,
\begin{equation}\label{keerolta432}
 J^*(t) \leq \frac{1}{t}  \int_0^t J^*(x)dx \leq \frac{C}{t} .
\end{equation}

Let us  now  fix $0 < \beta < p$ and  choose $\alpha \in (\beta,p)$. If we write
 $$ \phi(t) : =   \int_0^t J^*(x)dx, \quad  0 < t \leq \pi,
 $$
then $0 \leq \phi(t) \leq C  \log^{-\alpha}(e + 1/t),  0 < t \leq \pi$. In addition,  (\ref{keerolta432}) gives
$$\int_0^\pi J^*(x) \log^\beta \bigl(e +  J^*(x) \bigr) \leq C \int_0^\pi  \phi'(x) \log^\beta(e + 1/x) \, dx . 
$$
Now an integration by parts shows that the last integral is finite. We have thus proved  the claim (\ref{keerolta11}).
 
The second claim 
\begin{equation} \label{keerolta22}
\hskip-33pt|Df|^2 \, \log^{\beta -1}(e + |Df|) \in L^1_{loc}(\Omega) 
\end{equation}
 can be  deduced from (\ref{keerolta11}) by observing that
for every $\beta, p >0$ there are positive constants $C_1$ and $C_2$ such  that
$$
xy\log^{\beta -1}(e+\sqrt{xy})\leq C_1 \, x\log^{\beta}(e+x)+C_2 \, e^{p\, y} \quad \;\mbox{for all} \;\; x,y>0.
$$
For this consider  separately the cases where $x < e^{\frac{p}{2} y}$ and where  $x \geq e^{\frac{p}{2} y}$. Now, since $s \mapsto s^2 \log^{\beta - 1}(e + s)$ is increasing, we conclude

\begin{eqnarray}
&&\hskip-20pt|Df(z)|^2 \log^{\beta -1}\left( e + |Df(z)| \right)  \leq  K(z,f)  J(z,f) \log^{\beta -1}\left(e + \sqrt{K(z,f)J(z,f)} \right) \nonumber \\ \nonumber \\
&&\hskip60pt\leq C_1 J(z,f) \log^{\beta}\left( e + J(z,f)\right) + C_2\,  e^{p K(z,f) }  \in L^1_{loc}(\ID) .\nonumber
\end{eqnarray}
Lastly, the family $g_p$ from (\ref{kovexample}) shows that (\ref{keerolta11}), (\ref{keerolta22}) may fail  at the borderline  $\beta = p$. 
The proof of Theorem \ref{22.hkinov07} is complete.
\end{proof}

\section{Divergent Neumann series}\label{div}

We next   verify that our  decay estimate from Theorem \ref{2helsinkinov07} is essentially the best possible. One way to do this is simply to observe that if the decay of the $L^2$-norm of the $n$-th term
would   be of the order $O((n+1)^{-\delta\, p})$ with a uniform  $\delta >1/2$, our proof above would yield an area distortion result which would be  'too good', 
contradicting Example (\ref{kovexample}). We leave the details to the reader. Instead, we present 
here a more concrete approach that is of independent interest.

We start  by constructing families of  mappings. Suppose first  that $\gamma: (0,1] \to \IC$ is continuous with
$$|\gamma(s)| < 1  \quad \mbox{ for } 0 < s \leq 1$$
Then  define
\begin{equation} \label{kaava1}
\rho(t)  = \rho_{\gamma}(t): =\exp\left[- \int_t^1 \frac{1 + \gamma(s)}{1 - \gamma(s)} \,  \frac{ds}{s} \right]
\end{equation}
or equivalently,
\begin{equation}
\gamma(t) = \gamma_\rho(t) =  \frac{t\, \rho'(t) - \rho(t)}{t\, \rho'(t) + \rho(t)}, \quad\quad t \in   (0,1].
\nonumber
\end{equation}
In any case, we see from (\ref{kaava1}) that 
 $$ t \mapsto |\rho(t)|  \quad \mbox{ is strictly increasing on (0,1]}.
 $$
 If, in addition,
 \begin{equation}
\int_0^1 \Re e \left( \frac{1 + \gamma(s)}{1 - \gamma(s)} \right)\,  \frac{ds}{s} = \infty\nonumber
\end{equation}
 then $\rho(t) \to 0$ as $t \to 0$. In particular, in this case  the mapping
\begin{equation} 
 f(z) = \frac{z}{|z|} \rho(|z|), \quad |z|  \leq 1, \quad \mbox{ with } \quad f(z) =  z, \quad |z| \geq 1,\nonumber
\end{equation}
defines a homeomorphism of $\IC$.

Next, let $\alpha > 0$ with $$\hskip-30pt\rho(t) =  \left( \log \frac{5}{t} \right)^{-\alpha}  \quad \mbox{that is,}  \quad \gamma(t) = \frac{\alpha - \log(5/t)}{\alpha + \log(5/t)}.$$
 Consider then the  holomorphic family of  mappings
$$f_\lambda(z) = \frac{z}{|z|} \rho_{\lambda}(|z|), \quad \quad \rho_\lambda := \rho_{(\lambda \gamma)}, \;\mbox{ where }\;  |\lambda| < 1.$$  
In other words, $$\rho_{\lambda}(t) =\exp\left[ \int_1^t \frac{1 + \lambda \, \gamma(s)}{1 - \lambda \, \gamma(s)} \,  \frac{ds}{s} \right],  \;\;  0 < t \leq 1,  \quad \mbox{ and }  \quad \rho_{\lambda}(t) \equiv t,\;\;  t \geq 1.$$
The mappings $f_\lambda$ are all quasiconformal in the entire plane. In fact, they define a holomorphic motion 
$$\Phi: \ID \times \IC \to \IC, \quad \quad \Phi(\lambda,z) = f_\lambda(z).$$ 

At the boundary, when $\lambda \to \zeta \in \IS^1$, we attain well defined mappings $f_\zeta$ of finite distortion. If fact, calculating  the complex dilatation shows that $|\mu(z)| = |\gamma(t)|$ for $ |z|=t$. Hence the distortion function 
$$K(z, f_\zeta) =  \frac{1 + | \gamma(t)|}{1 -  | \gamma(t)|} , \quad t = |z| \mbox{ and } |\zeta| = 1,$$ 
which  is independent of $\zeta$, with
 $$K \simeq \frac{1}{\alpha} \log \frac{1}{|z|} \quad \quad \mbox{ as } |z| \to 0.
 $$
 Thus given $\varepsilon > 0$ we can choose $\alpha < \frac{1}{2}$  so that $e^{ K(z)}  \in L^{1-\varepsilon}(\ID)$. On the other hand, one computes that
 $$\frac{\partial f_\zeta}{\partial \zbar}(z) \simeq  \frac{1}{|z|}  \left( \log \frac{1}{|z|} \right)^{\frac{-4\alpha }{|1+\zeta|^2}} \quad \quad \mbox{ as } |z| \to 0.
 $$
Therefore  $\partial_\zbar f_\zeta \notin L^{2}(\IC)$ whenever $\zeta$ belongs to the non-degenerate interval
 $$ \{ \zeta \in \IS^1 : 2\alpha < \frac{|1+\zeta|^2}{4} \} .
 $$
In particular, $\lambda \to \partial_\zbar f_\lambda$ does not belong to the Hardy space $H^2(\ID; L^2(\IC))$, and hence the sequence of norms $\| (\mu {\mathcal S})^n \mu\|_2$ is not square summable, i.e.
$$\| (\mu {\mathcal S})^n \mu\|_2 \not\leq C n^{-(1+\varepsilon)/2}
$$ 
and the decay given by Theorem \ref{2helsinkinov07} cannot be improved at $p=1$. 
 \bigskip

 \section{Applications to Degenerate Elliptic PDE's }\label{elliptic}

 The optimal regularity established in Theorem \ref{22.hkinov07}  
 obviously has a number of basic consequences, for instance towards
  removability of singularities.  However, our aim here is not to  
 consider these consequences systematically. We only indicate 
 one such example, namely an application to degenerate elliptic  
 equations.
  Degenerate elliptic PDE's arise naturally in hydrodynamics,  
 nonlinear elasticity,  holomorphic dynamics  and several other  
 related areas. In two dimensions these equations are intimately  
 related to the mappings of finite distortion.

Suppose that $u \in  W^{1,1}_{loc}(\Omega)$ is a  
 (distributional) solution to the equation
 \begin{equation} \label{divtypeyht}
  \Div\,  A(z) \nabla u = 0, \quad \quad \mbox{ in } \Omega \subset  
 \IR^2.
 \end{equation}
 Let us  assume that   $u$ has {\it finite energy}, that is
 \begin{equation}  \label{divtypeyhtener}
 \int_\Omega \langle \nabla u(z) , A(z) \nabla u (z) \rangle \; dm 
 (z) < \infty.
 \end{equation}
 In case the  domain $\Omega$ is simply connected, the solution $u$   
 admits a conjugate function, 
  defined by the  
 condition
 \begin{equation} \label{lisakaava15}
 \nabla v = J A(z) \nabla u.
 \end{equation}
 Here $J(x,y) = (-y,x)$ is the rotation by a right angle.

 In the complex notation,  the function
 $$ f = u + iv
 $$
 solves the $\IR$-linear equation
 \begin{equation} \label{rlineqkaava}
 \frac{\partial f}{\partial \zbar} = \mu(z)  \frac{\partial f} 
 {\partial z} + \nu(z)\,  {\overline{\frac{\partial f}{\partial z}}} 
 \, ,
 \end{equation}
 where the coefficients $\mu(z)$,  $\nu(z)$ are explicit functions   
 of the elements of $A(z)$, independent  of the  particular solution  
 $u$.  Conversely, for  these coefficient $\mu$ and $\nu$, the real  
 part of any solution $f = u + iv$ to (\ref{rlineqkaava}) solves the  
 second order  equation (\ref{divtypeyht}).
 \medskip

 Next,  let us assume that for almost all $z$, the matrix $A(z)$ is  
 elliptic and  symmetric. Then there is a function $1 \leq K(z) <  
 \infty$, finite for almost all $z$, such that
 \begin{equation} \label{lisakaava16}
  \frac{1}{K(z)} |h|^2 \leq  \langle h , A(z) h \rangle \leq K(z)  | 
 h|^2, \quad h \in \IR^2.
 \end{equation}
 In terms of the mapping $f = u + iv$, the conditions (\ref 
 {lisakaava15}), (\ref{lisakaava16}) take the familiar form
 $$ |Df(z)|^2 \leq K(z) J(z,f),
 $$
 where the Jacobian can alternatively be written as
 $$ J(z,f) =  \langle \nabla u(z) , A(z) \nabla u (z) \rangle .
 $$
 In particular, $f = u + iv$ is a mapping of finite distortion if   
 and only if the solution $u$ has finite energy  in the sense (\ref 
 {divtypeyhtener}). As an immediate consequence of Corollary \ref 
 {l2tapaus} we find optimal conditions guaranteeing the $L^2$-regularity  
 of the derivatives of  the solutions $u$.

  \begin{theorem} \label{hotelli83}
  Suppose $u \in W^{1,1}_{loc}(\Omega)$ is a  solution  to  the   
 equation {\rm  (\ref{divtypeyht})} with
  $$\int_\Omega \langle \nabla u(z) , A(z) \nabla u (z) \rangle \; dm 
 (z) < \infty.$$
 If the ellipticity bound $K(z)$ of the coefficient matrix $A(z)$  
 satisfies
 \begin{equation} \label{hotelli71}
  e^{K(z)}  \in L^p_{loc}(\Omega)
 \end{equation}
 for some  $p > 1$,  then the solution $u$ has the regularity
 \begin{equation}
 \label{hotelli90}
   |\nabla u|^2 +  |{A}(z) \nabla u)|^2  \in L^1_{loc}(\Omega).
 \end{equation}
 \end{theorem}

 \bigskip

 \noindent Similarly, we have
 $$ |\nabla u|^2 \log^{\beta -1}(e + |\nabla u|) \in L^1_{loc} 
 (\Omega), \quad \quad 0 < \beta < p,
 $$
 and
 $$
  |{ A}(z)\nabla u)|^2 \log^{\beta -1}(e + |{A}(z)\nabla u)|)  
 \in L^1_{loc}(\Omega),
  \quad \quad 0 < \beta < p
 $$
 \smallskip

 \noindent whenever the ellipticity $K(z)$ satisfies $e^{K(z)} \in  
 L^p_{loc}(\Omega)$, $p > 0$.
 Furthermore, via the Stoilow factorization many properties of  the 
 harmonic functions generalize to the solutions to these degenerate  
 elliptic equations. There are even non-linear counterparts, see \cite{AIM}.

\bigskip

One of the fundamental features  of  the mappings of bounded distortion are their self-improving regularity properties. Similar phenomena occur also in the theory of  mappings of finite distortion. A typical example is obtained by combining the Stoilow factorization Theorem \ref{genstoilow} with our Theorem  \ref{22.hkinov07}. There are factorization theorems slightly beyond the class  $W^{1,Q}_{loc}(\IC)$, $Q(t)=t^2 \log^{-1}(e+t)$, see for instance \cite{IMbook}, p. 280. Both factorization methods lead to removability results for mappings of finite distortion. We refer the reader to  \cite{IMbook}, \cite{FKZ}, \cite{AIKM} and their references for   these and other related  methods, such as Caccioppoli-type estimates.  In view of the  discussion in this section, one obtains corresponding removability results   for the  solutions of the  elliptic PDE's.
\bigskip

\noindent {\bf Acknowledgements}. The authors want to thank Albert Clop for pointing out a mistake in an earlier version of  Example  (\ref{kovexample}).

\end{document}